\documentclass[11pt]{amsart}

\usepackage{float}

\usepackage{amsmath, amsthm, color, fullpage}
\usepackage{graphicx, amsaddr}
\usepackage{enumerate}

\newtheorem{theorem}{Theorem}[section]
\newtheorem{corollary}[theorem]{Corollary}

\newtheorem{lemma}[theorem]{Lemma}

\title{Spanning tree packing and 2-essential edge-connectivity}
\author{Xiaofeng Gu$^{1}$, Runrun Liu$^{2}$,      Gexin Yu$^{3}$}

\address{
$^{1}$\small Department of Computing and Mathematics, University of West Georgia, Carrollton, GA 30118, USA.\\
$^{2}$\small School of Mathematics, Zhejiang Normal University, Jinhua, Zhejiang 321004, China.\\
$^3$\small Department of Mathematics, William \& Mary, Williamsburg, VA 23185, USA.
}

\thanks{The research of the first author was supported by a grant from the Simons Foundation (522728, XG). The research of the second author was supported in part by NSFC(12101563), ZJNSFC(LQ22A010011) and the China Postdoctoral Science Foundation (2022M712834).}

\email{xgu@westga.edu, rliu1206@zjnu.edu.cn, gyu@wm.edu}



\date{\today}

\begin{document}

\maketitle

\begin{abstract}
An edge (vertex) cut $X$ of $G$ is $r$-essential if $G-X$ has two components each of which has at least $r$ edges. A graph $G$ is $r$-essentially $k$-edge-connected (resp. $k$-connected) if it has no $r$-essential edge (resp. vertex) cuts of size less than $k$. If $r=1$, we simply call it essential. Recently, Lai and Li proved that every $m$-edge-connected essentially $h$-edge-connected graph contains $k$ edge-disjoint spanning trees, where $k,m,h$ are positive integers such that $k+1\le m\le 2k-1$ and $h\ge \frac{m^2}{m-k}-2$. In this paper, we show that every $m$-edge-connected and $2$-essentially $h$-edge-connected graph that is not a $K_5$ or a fat-triangle with multiplicity less than $k$ has $k$ edge-disjoint spanning trees, where $k+1\le m\le 2k-1$ and $$h\ge f(m,k)=\begin{cases}
2m+k-4+\frac{k(2k-1)}{2m-2k-1}, & m< k+\frac{1+\sqrt{8k+1}}{4}, \\
m+3k-4+\frac{k^2}{m-k}, &  m\ge k+\frac{1+\sqrt{8k+1}}{4}.
\end{cases}$$
Extending Zhan's result, we also prove that every 3-edge-connected essentially 5-edge-connected and $2$-essentially 8-edge-connected graph has two edge-disjoint spanning trees. As an application, this gives a new sufficient condition for Hamilton-connectedness of line graphs. In 2012, Kaiser and Vr\'ana proved that every 5-connected line graph of minimum degree at least 6 is Hamilton-connected. We allow graphs to have minimum degree 5 and prove that every 5-connected essentially 8-connected line graph is Hamilton-connected.
\end{abstract}

\section{Introduction}
We consider finite graphs with possible parallel edges but no loops. The {\it multiplicity} of a graph is the maximum number of edges between any pair of vertices in the graph.
The famous packing spanning trees theorem by Tutte~\cite{Tutte61} and independently Nash-williams~\cite{Nash61} implies that $2k$-edge-connected graphs contain $k$ edge-disjoint spanning trees.  As a graph with $k$ edge-disjoint spanning trees may contain vertices of degree less than $2k$, people have tried to explore other conditions to pack spanning trees.  One of them is from Nash-Williams~\cite{Nash64}. Note that a graph is nontrivial if it contains at least one nonloop edge.

\begin{theorem}[Nash-Williams~\cite{Nash64}]\label{thm:nash}
If $|E(G)|\ge k(|V(G)|-1)$, then $G$ has a nontrivial subgraph with $k$ edge-disjoint spanning trees.
\end{theorem}

By this theorem, if a graph $G$ contains a subgraph $H$ with $|E(H)|\ge k(|V(H)|-1)$, then $H$ contains a nontrivial subgraph $H_1$ with $k$ edge-disjoint spanning trees. Let $X$ be a subgraph of a graph $G$. The contraction $G/X$ is the graph obtained from $G$ by identifying the two ends of each edge in $X$ and then deleting the resulting loops. If $G/H_1$ also satisfies $|E(G/H_1)|\ge k(|V(G/H_1)|-1)$, then the $k$ edge-disjoint spanning trees from $G/H_1$ combined with the spanning trees from $H_1$ give $k$ edge-disjoint spanning tree for $G$.   Continue this contraction process as long as $|E(H)|\ge k(|V(H)|-1)$ for some subgraph $H$, we obtain a reduced graph $G'$, where a graph is reduced if each nontrivial subgraph $H$ satisfies $|E(H)|<k(|V(H)|-1)$. Now $G$ has $k$ edge-disjoint spanning trees as long as the reduced graph $G'$ has $k$ edge-disjoint spanning trees. Therefore, if a graph does not have $k$ edge-disjoint spanning trees, its reduced graph must be in the family $$\mathcal{G}=\{G: \text{ for each nontrivial subgraph } H\subseteq G, |E(H)|< k(|V(H)|-1)\}.$$ All graphs in $\mathcal{G}$ are sparse, and as a matter of fact, they all have maximum average degree less than $2k$. To get a condition for packing spanning trees, one only needs the condition to get rid of this sparseness.  Now it is clear that the condition ``$2k$-edge-connected'' works, since this condition gives rise of minimum degree $2k$ in the reduced graph after contracting subgraphs.

To maintain minimum degree $2k$ in the reduced graph, weaker conditions than being $2k$-edge-connected might also work.  One such option is the so-called essential edge-connectivity, first introduced by Chartrand and Stewart~\cite{CS69} to study the connectivity of line graphs. For a vertex subset or a subgraph $X$ of $G$, $G-X$ denotes the graph obtained from $G$ by deleting all vertices in $X$ and their incident edges. A graph $G$ is {\em essentially $k$-edge-connected} if $G$ has no edge cut $X$ of size less than $k$ such that $G-X$ has two nontrivial components.
Such a result has recently been obtained by Lai and Li~\cite{LL19}.

\begin{theorem}[Lai and Li~\cite{LL19}]\label{LL19}
Let $k,m$ be integers with $k+1\le m\le 2k-1$. Let $G$ be an $m$-edge-connected essentially $h$-edge-connected graph. If $$h\ge \frac{m^2}{m-k}-2=m+k+\frac{k^2}{m-k}-2,$$
then $G$ contains $k$ edge-disjoint spanning trees.
\end{theorem}

In this paper, we use {\em sum of degrees of adjacent vertices} (this is called Ore-degree by some authors) to break the sparseness in the reduced graph to make further improvement. To maintain such a degree sum, we naturally use the following extension of essential edge-connectivity introduced in \cite{KV21}. An edge cut $X$ of $G$ is {\em $r$-essential} if $G-X$ has two components each of which has at least $r$ edges. A graph $G$ is {\em $r$-essentially $k$-edge-connected} if it has no $r$-essential edge cuts of size less than $k$.  Clearly $1$-essentially $k$-edge-connected is equivalent to {\em essentially $k$-edge-connected}. Li and Yang~\cite{LY12} studied two edge-disjoint spanning trees by using 2-essential edge-connectivity. We obtain the following general result. Note that our bound on $h$ is slighter better than twice of the bound in Theorem~\ref{LL19}. A {\it fat-triangle} is a multigraph whose underlying simple graph is a $K_3$.

\begin{theorem}\label{main2}
Let $k+1\le m\le 2k-1$. Let $G$ be an $m$-edge-connected 2-essentially $h$-edge-connected graph that is not a $K_5$ or a fat-triangle with multiplicity at most $k-1$. If \begin{equation}\label{eq-h}
h\ge f(m,k)=\begin{cases}
2m+k-4+\frac{k(2k-1)}{2m-2k-1}, & m< k+\frac{1+\sqrt{8k+1}}{4}, \\
m+3k-4+\frac{k^2}{m-k}, &  m\ge k+\frac{1+\sqrt{8k+1}}{4},
\end{cases}
\end{equation}
then $G$ has $k$ edge-disjoint spanning trees.
\end{theorem}

Note that $K_5$ is a counterexample only for $m=k+1=4$, and not all fat-triangles with multiplicity at most $k-1$ are counterexamples. We feel it is not interesting to characterize exact counterexamples on three vertices. 

\medskip
For $k=2$, Zhan~\cite{Zhan91} implicitly proved that every $3$-edge-connected essentially $7$-edge-connected graph has two edge-disjoint spanning trees.  We can improve Zhan's result by decreasing essential edge-connectivity to guarantee the existence of two edge-disjoint spanning trees. 

\begin{theorem}\label{main1}
Every $3$-edge-connected essentially $5$-edge-connected and $2$-essentially $8$-edge-connected graph has two edge-disjoint spanning trees.
\end{theorem}

A graph is {\em Hamiltonian} if it contains a Hamilton cycle, and is {\em Hamilton-connected} if every pair of its vertices is joined by a Hamilton path.  A long-standing conjecture made by Thomassen~\cite{Thom86} states that every $4$-connected line graph is Hamiltonian.  Zhan's result~\cite{Zhan91} implies that every 7-connected line graph is Hamiltonian, making a progress towards Thomassen's conjecture. In fact, his proof implies that the graph is also Hamilton-connected. A graph $G$ is {\em essentially $k$-connected} if $G$ has no vertex cut $X$ of size less than $k$ such that $G-X$ has two nontrivial components. Extending Zhan's result, Lai et al.~\cite{LSWZ06} proved that every 3-connected essentially 11-connected line graph is Hamiltonian. 
This result was subsequently improved in \cite{YLLG12, LY12}.

The current best results are due to Kaiser and Vr\'ana. They proved that every $5$-connected line graph of minimum degree at least $6$ is Hamilton-connected in~\cite{KV12}, and every $3$-connected essentially $9$-connected line graph is Hamilton-connected in~\cite{KV21}. Theorem~\ref{main2} gives us the following corollary, which allows graphs to have minimum degree $5$, thus provides a valuable addition to the theorems of Kaiser and Vr\'ana~\cite{KV12, KV21}.

\begin{corollary}\label{cor:line}
Every 5-connected essentially 8-connected line graph is Hamilton-connected.
\end{corollary}

We may point out that, by using the closure concept of Ryj\'a\v cek and Vr\'ana~\cite{RV11}, Corollary~\ref{cor:line} can be generalized from line graphs to claw-free graphs.

\medskip
Lai and Li~\cite{LL19} mentioned several other applications of graphs with $k$ edge-disjoint spanning trees, such as nowhere-zero $3$-flow, circular flow, spanning connectivity of line graphs and supereulerian width of graphs. Our result gives analogue results on those properties as well.

In the end of this section, we introduce some notations used in the paper.
If $X, Y$ are disjoint vertex subsets or subgraphs of $G$, then $E(X, Y)$ and $e(X, Y)$ denote the set and the number of edges with one end in $X$ and the other end in $Y$, respectively. We use $e(G)$ for $|E(G)|$. If $X=\{x\}$ and $Y=\{y\}$, we also write $e(X,Y)$ as $e(xy)$, which is the number of edges between $x$ and $y$.
A $k$-vertex (resp. $k^+$-vertex, $k^-$-vertex) is a vertex of degree $k$ (resp. at least $k$, at most $k$). Similarly, let $u$ be a neighbor of a vertex $v$. We call $u$ a $k$-neighbor (resp. $k^+$-neighbor, $k^-$-neighbor) of $v$ if $u$ has degree $k$ (resp. at least $k$, at most $k$).

\section{Proof of Theorem~\ref{main2}}

Let $\mathcal{G}$ denote the family of $m$-edge-connected and $2$-essentially $h$-edge-connected graphs where $h$ is given as in \eqref{eq-h}, that are not fat-triangles with multiplicity at most $k-1$ or $K_5$. Since $$m+3k-4+\frac{k^2}{m-k}=4k-4+(m-k)+\frac{k^2}{m-k}>4k-4+2k= 6k-4$$ and
$$(2m+k-4+\frac{k(2k-1)}{2m-2k-1})- (m+3k-4+\frac{k^2}{m-k})
=\frac{2(m-2k)(m-k-\frac{1+\sqrt{8k+1})}{4})(m-k+\frac{\sqrt{8k+1}-1}{4})}{(2m-2k-1)(m-k)},$$
we have
\begin{equation}\label{h-max}
h\ge \max\{6k-4, 2m+k-4+\frac{k(2k-1)}{2m-2k-1}, m+3k-4+\frac{k^2}{m-k}\}.
\end{equation}

Suppose that $G\in \mathcal{G}$ is a counterexample to Theorem~\ref{main2} such that $|E(G)|$ is as small as possible. Then $G$ has no $k$ edge-disjoint spanning trees but any graph $G'\in \mathcal{G}$ with $|E(G')|<|E(G)|$ has $k$ edge-disjoint spanning trees.

By Theorem~\ref{thm:nash}, if $|E(G)|\ge k(|V(G)|-1)$, then $G$ has a nontrivial subgraph $H$ with $k$ edge-disjoint spanning trees. Notice that $G/H$ is still an $m$-edge-connected and $2$-essentially $h$-edge-connected graph. If $G/H$ is a $K_5$, then the original graph $G$ would have a $2$-essential edge cut of size at most $4<h$, and likewise, if $G/H$ is a fat-triangle of multiplicity at most $k-1$, we also obtain a $2$-essential edge cut of size at most $2k<h$ in $G$, violating the assumption of $G$.
Thus $G/H\in \mathcal{G}$, and by the minimality of $G$, $G/H$ has $k$ edge-disjoint spanning trees. Both $H$ and $G/H$ have $k$ edge-disjoint spanning trees, and this implies that $G$ also has  $k$ edge-disjoint spanning trees, a contradiction.

Thus we have $|E(G)|< k(|V(G)|-1)$, and so $|E(G)|\le k|V(G)|-k-1$. This implies that
\begin{equation}\label{sum-of-degrees}
\sum_{v\in V(G)}(d(v)-2k)\le -2k-2.
\end{equation}

\begin{lemma}\label{cut-claim}
Let $uv\in E(G)$ with $d(u)+d(v)<\frac{m^2}{m-k}$. Then for each $w\in N(u)\cup N(v)\setminus \{u,v\}$, either $d(w)\ge 2k$ or at least one component in $G-\{u,v,w\}$ has at least two edges. \end{lemma}

\begin{proof}
By symmetry of $u,v$, suppose otherwise that for some $w\in N(v)-u$,  $d(w)<2k$ and each component in $G-\{u,v,w\}$ has at most one edge. 

Let $T=\{u,v,w\}$. We first prove the following claim.

\begin{quote}
{\bf Claim:} for any vertex $x\not\in T$, if $e(xy)\ge 2$ for some $y\in T$, then $d(x)+d(y)\ge h+4.$
Consequently, $d(x)\ge h+4-d(y)\ge h+4-\max\{\frac{m^2}{m-k},2k\}\ge 2k$.
\end{quote}

{\bf Proof of the claim:}
Suppose otherwise that $d(x)+d(y)< h+4$. Then $E(\{x,y\}, G-x-y)$ is an edge cut of size less than $h$, which implies that each component in $G-x-y$ contains at most one edge. So each $z\in V(G-x-y)$ has at least $d(z)-1$ edges to $\{x,y\}$. Furthermore, each $z\in V(G-T-x)$ has at least $d(z)-1\ge m-1\ge k$ edges to $y$, since at most one vertex of $G-T-x$ is adjacent to $x$.

Let $G_1=G[T\cup\{x\}]$. First of all, we observe that there are at least $k$ edge-disjoint spanning stars $T_1,T_2,\cdots, T_k$ centered at $y$ in $G - (G_1-y)$. Note that each vertex in $G_1-y$ has at most one edge to $G-G_1$ since $e(\{x,y\}, G-x-y)<h$ and $e(T,G-T)<h$.

First we consider that each vertex in $G_1-y$ has no edges to $G-G_1$. As $e(\{x,y\},G_1-x-y)<h$ and $e(T,G-T)<h$, if there are multi-edges in $G_1$, then all multi-edges must share an end vertex.
Therefore, there exists a vertex $z\in \{x,y\}$ such that each vertex in $G_1-z$ has at most one edge not to $z$. So $e(G-z)\le3$. If $e(G-z)\le1$, then each vertex in $G_1-z$ has at least $m-1$ edges to $z$. Then we can obtain  $k$ edge-disjoint spanning stars centered at $z$ in $G_1$.  If $e(G-z)=2$, then each vertex in $G_1-z$ has at least $m-2$ edges to $z$. Then we can obtain $k-1$ edge-disjoint spanning stars $S_1,S_2,\cdots,S_{k-1}$ of $G_1$ centered at $z$ in $E(z,G_1-z)$ and get the $k$-th spanning tree of $G_1$ in $E(G_1)\setminus \bigcup_{1\le i\le k-1} S_i$. If $e(G-z)=3$, then each vertex in $G_1-z$ has at least $m-2$ edges to $z$. Then we can obtain  $k-2$ edge-disjoint spanning stars $S_1,S_2,\cdots,S_{k-2}$ of $G_1$ centered at $z$ in $G_1-K_4$ and get the $(k-1)$-th and $k$-th spanning trees of $G_1$ in $K_4$. In all cases, we can obtain $k$ edge-disjoint spanning trees in $G_1$, which together with the $k$ edge-disjoint spanning stars in $G- (G_1-y)$, give $k$ edge-disjoint spanning trees in $G$, a contradiction.

So we may assume that some vertices in $G_1-y$ have one edge to $G-G_1$. By the above argument, we can obtain $k-1$ edge-disjoint spanning trees $S_1,S_2,\cdots, S_{k-1}$ in $G_1$ that contains $Z=\{z:z\in V(G_1-y) \text{ and } z \text{ has one edge } zz' \text{ to } G-G_1\}$ and one spanning tree $S_k$ in $G_1-Z$ with   $E(S_k)\subset E(G_1)\setminus \bigcup_{1\le i\le k-1} S_i$. So we can get $k$ spanning trees $S_1\cup T_1, S_2\cup T_2, \cdots, S_{k-1}\cup T_{k-1}, S_k\cup T_k\cup \{zz':z\in Z\}$ in $G$, a contradiction.
Thus $d(x)+d(y)\ge h+4$,  completing the proof of the claim. $\Box$


\




If $|V(G)|\le3$, then $G$ must be a fat-triangle of multiplicity at most $k-1$. For otherwise, we can find $k$ edge-disjoint spanning trees in $G$ since $G$ is $m$-edge-connected, a contradiction. So we may assume that $|V(G)|\ge4$.

First assume that $m\ge 5$.  Then each vertex $x\not\in T$ has at least $m-1\ge 4$ edges to $T$, thus must have at least two edges to some vertex in $T$. By the Claim above, $d(x)\ge 2k$ for each $x\not\in T$. Let $x\not\in T$ and $y\in T$ with $e(xy)\ge 2$. Then
\begin{align*}
\sum_{z\in V(G)} (d(z)-2k)&\ge \sum_{z\in T} (d(z)-2k)+(d(x)-2k)\ge (d(x)+d(y)-4k)+\sum_{z\in T-y} (d(z)-2k)\\
&\ge (h+4-4k)+(2m-4k)\ge (6k-4+4-4k)+(2(k+1)-4k)=2,
\end{align*}
a contradiction to \eqref{sum-of-degrees}.

So we may assume that $m\le 4$. Since $k+1\le m<2k$, we have $(m,k)=(4,3)$ or $(m,k)=(3,2)$. As each vertex $z\not\in T$ has at most one neighbor in $V(G)-T$, $e(T, V(G)-T)\ge 3(n-3)-2e(G-T)$. So $k(n-1)>e(G)\ge m(n-3)-2e(G-T)+e(G-T)+2$. This implies that $e(G-T)>m(n-3)-k(n-1)+2$. From $e(G-T)\le \lfloor(n-3)/2\rfloor$,  we have
\begin{equation}\label{n}
m(n-3)-k(n-1)+2<\lfloor(n-3)/2\rfloor.
\end{equation}

Let $m=3$ and $k=2$. Then by \eqref{n}, $n\in\{4,5\}$. If $n=4$, then $e(G)\ge \frac{mn}{2}=6\ge 2(n-1)$. If $n=5$, then one vertex must have even degree by the Handshaking lemma, and thus $e(G)\ge \sum_{z\in V(G)} d(z)/2\ge (4+3(n-1))/2\ge 2(n-1)$. Both reach contradictions to \eqref{sum-of-degrees}.

Let $m=4$ and $k=3$.  Then by \eqref{n}, $n\le 9$. Note that $h=22$ and $e(G)\le 3(n-1)-1\le 23$. Recall that $uvw$ is a path in $G$. First we show that $u$ or $w$ has at most one edge to $V(G)-T$. By symmetry say $u$ has two edges to $z_1, z_2\notin T$ (it may happen that $z_1=z_2$), then $v,w$ have no neighbors in $V(G)-T-z_1-z_2$ and $e(vw)=1$, for otherwise, $E(\{u,z_1, z_2\}, V(G)-\{u,z_1,z_2\})$ is a $2$-essential edge cut of size less than $h$ since $e(G)\le 23$. Then each $z\in V(G)-T-z_1-z_2$ has at least $d(z)-1\ge m-1=3$ edges to $u$.

If $V(G)-T\not=\{z_1, z_2\}$, then choose $z\in V(G)-T-z_1-z_2$. Then none of $v,w$ is adjacent to $z_1$ or $z_2$, for otherwise, $E(\{z,u\}, V(G)-\{z,u\})$ is a $2$-essential edge cut of size less than $h$ since $e(G)\le 23$. So each vertex in $\{v,w,z_1,z_2\}$ has at least $m-1$ edges to $u$. Now we can obtain $3$ edge-disjoint spanning stars centered at $u$ in $G$, a contradiction. So we may assume that  $V(G)-T=\{z_1, z_2\}$. If $z_1=z_2$, then $w$ has at least three edges to $\{u,z_1\}$ and $v$ has at least two edges to $\{u,z_1\}$  since $\delta(G)\ge 4$ and $e(vw)=1$. This implies that $\Delta(G)\ge 5$. So $e(G)\ge \frac{4+4+4+5}{2}\ge 9 \ge 3(n-1)$, a contradiction to \eqref{sum-of-degrees}. If $z_1\ne z_2$, then $G$ contains no parallel edges, for otherwise let $e(xy)\ge2$ for any two vertices $x,y$ in $G$. Then $G-x-y$ contains at most one edge. So $e(G)\ge e(xy)+e(G-x-y)+e(\{x,y\},G-x-y)\ge 2+1+3+3+4=13> 3(n-1)$, a contradiction to \eqref{sum-of-degrees}. Since $\delta(G)\ge 4$, $G$ must be $K_5$.

Now we have each of $u$ and $w$ has at most one edge to $V(G)-T$. It follows that each vertex in $G-v$ has at least $m-1=3$ edges to $v$, and again we can obtain $3$ edge-disjoint spanning stars centered at $v$ in $G$, a contradiction.
This completes the proof of Lemma~\ref{cut-claim}.
\end{proof}

We use the discharging method to show that $\sum_{v\in V(G)}(d(v)-2k)\ge0$ to reach a contradiction. Let $v\in V(G)$ have an initial charge of $\mu(v)=d(v)-2k$.
We design some discharging rules and redistribute weights accordingly.
Let $\mu^*(v)$ be the final charge after the discharging procedure.

The discharging rules are designed as follows:

\begin{enumerate}
    \item[(R1)] Each $d$-vertex with $d\ge2k$ distributes its surplus $d-2k$ evenly to all its neighbors of degree at most $2k$.
\end{enumerate}

Let $v$ be a vertex in $G$. If $d(v)\ge2k$, then by (R1) $\mu^*(v)\ge d(v)-2k-\frac{d(v)-2k}{d(v)}\cdot d(v)=0$. So we may assume that $m\le d(v)\le 2k-1$.  If each neighbor of $v$ has degree at least $t=\frac{kd(v)}{d(v)-k}$, then by (R1) $v$ gets at least $\frac{t-2k}{t}=\frac{2k-d(v)}{d(v)}$ from each neighbor. Thus $\mu^*(v)\ge d(v)-2k+\frac{2k-d(v)}{d(v)}\cdot d(v)=0$. So we may assume that $v$ has a neighbor $u$ with degree less than $t$.

First we show that $u$ has at most one neighbor with degree less than $2k$. Suppose otherwise that $u$ has a $(2k-1)^-$-neighbor $w$ distinct from $v$. Note that $d(v)+d(u)<d(v)+\frac{kd(v)}{d(v)-k}\le m+\frac{km}{m-k}=\frac{m^2}{m-k}$, where we use the property that $f(x)=\frac{kx}{x-k}+x$ is decreasing on $m\le x\le 2k$. So by Lemma~\ref{cut-claim} $G$ has an edge cut $X$ of size at most $|X|=d(v)+d(u)+d(w)-4$ such that at least two components of $G-X$ each of which has at least two edges.  However $$|X|\le \frac{kd(v)}{d(v)-k}+d(v)+2k-1-4\le\frac{km}{m-k}+m+2k-5=\frac{k^2}{m-k}+m+3k-5< h,$$ contrary to that $G$ is 2-essentially $h$-edge-connected.

Therefore by (R1), $u$ would give all its surplus to $v$ if $2k+1\le d(u)<t$, which is at least $1$ and more than $\frac{t-2k}{t}$, the amount we calculated above for $v$ to get from a vertex of degree at least $t$. So let $u$ be a neighbor of $v$ with degree at most $2k$.  Note that $d(v),d(u)\ge m$ since $G$ is $m$-edge-connected. So let $d(u)+d(v)=d$, where $2m\le d\le 4k-1$. Note that $\frac{m^2}{m-k}-4k=\frac{(m-2k)^2}{m-k}>0$ since $k+1\le m\le 2k-1$. So $d<\frac{m^2}{m-k}$. By Lemma~\ref{cut-claim}, for each neighbor $w$ of $u$ and $v$, either $d(w)\ge2k$ or at least one component in $G-\{u,v,w\}$ has at least two edges. In the latter case, since $G$ is 2-essentially $h$-edge-connected, $w$ has degree at least $h-(d-3)+1=h-d+4.$ So we have $d(w)\ge \min\{2k,h-d+4\}=2k$. Now it is enough to show that $\mu^*(v)+\mu^*(u)\ge0$. By (R1), each of $u$ and $v$ gets at least $\frac{(h-d+4)-2k}{h-d+4}$ from each neighbor with degree at least $h-d+4$. So it suffices to have the following:
$$\mu^*(u)+\mu^*(v)\ge d-4k+\frac{(h-d+4)-2k}{h-d+4}(d-2)\ge 0.$$

From above, it suffices to have  $$h\ge d-4+\frac{k(d-2)}{d-2k-1}=d+k-4+\frac{k(2k-1)}{d-2k-1}.$$

Now let $g(x)=x+k-4+\frac{k(2k-1)}{x-2k-1}$. Since $$g'(x)=1-\frac{k(2k-1)}{(x-2k-1)^2}=\frac{(x-2k-1+\sqrt{k(2k-1})(x-2k-1-\sqrt{k(2k-1)})}{(x-2k-1)^2},$$ $g$ is decreasing on $2k+1-\sqrt{k(2k-1)}\le 2m\le x\le 2k+1+\sqrt{k(2k-1)}$ and increasing on $2k+1+\sqrt{k(2k-1)}\le x\le 4k-1$. So $$ \max\{g(2m), g(4k-1)\}=\max\{2m+k-4+\frac{k(2k-1)}{2m-2k-1},\ 6k-5+\frac{k}{2k-2}\}\le h.$$
This competes the proof of Theorem~\ref{main2}.
$\hfill\square$

Theorem~\ref{main2} indicates that every 3-edge-connected 2-essentially 10-edge-connected graph has two edge-disjoint spanning trees. By adding an extra essential 5-edge-connectivity, the 2-essentially edge-connectivity can be decreased from 10 to 8. This will be proved in the next section.

\section{Proofs of Theorem~\ref{main1} and Corollary~\ref{cor:line}}

Let $\mathcal{G}$ denote the family of 3-edge-connected essentially 5-edge-connected and 2-essentially 8-edge-connected graphs. Let $G\in \mathcal{G}$ be a counterexample to Theorem~\ref{main1} such that $|E(G)|$ is as small as possible. We first prove the following lemma.

\begin{lemma}\label{lem}
The following is true about $G$:
\begin{enumerate}[(i)]
\item $\delta{(G)}\ge3$;
\item two $3$-vertices cannot be adjacent;
\item let $u,v,w\in V(G)$ such that $u,w$ are two neighbors of $v$. Then $d(u)+d(v)+d(w)\ge12$.
\end{enumerate}
\end{lemma}

\begin{proof}
(i) Since $G$ is $3$-edge-connected, $\delta{(G)}\ge3$.

(ii) Suppose otherwise that $u,v$ are two adjacent $3$-vertices in $G$. Let $e_1,e_2,e_3,e_4$ be the four edges incident to $uv$. Since $\delta{(G)}\ge3$ by (i), $X=\{e_1,e_2,e_3,e_4\}$ is a $4$-edge cut such that at least two component of $G-X$ each of which has at least one edge, contrary to $G$ is essentially $5$-edge-connected.

(iii) Suppose otherwise that $d(u)+d(v)+d(w)\le11$.
Let $X$ be the set of edges with one end in $\{u,v,w\}$ and the other end in $V(G)-\{u,v,w\}$.
Then $|X|\le11-4=7$. We claim that $X$ is a $2$-essential edge cut, which would contradict the condition that $G$ is 2-essentially $8$-edge-connected. Suppose not. Let $\ell=|V(G)-\{u,v,w\}|$. Then $G-\{u,v,w\}$ contains no component of more than one edge, which implies that there are at most $\ell/2$ edges in $G-\{u,v,w\}$.
We now count the number of edges between $\{u,v,w\}$ and $V(G)-\{u,v,w\}$, and it follows that $$7\ge |X|\ge \sum_{x\not=u,v,w} d(x)-\ell.$$
So $3\ell\le \sum_{x\not=u,v,w} d(x)\le 7+\ell$ and we have $\ell\le 3$. When $\ell=3$, there are at least eight edges each of which has exactly one end in $\{u,v,w\}$ since $\delta(G)\ge 3$ and two $3$-vertices cannot be adjacent by Lemma~\ref{lem}(i)(ii), thus $d(u)+d(v)+d(w)\ge 8+4=12$, as desired.  If $\ell=2$, then clearly it has two edge-disjoint spanning trees, which implies the truth of Theorem~\ref{main1}.
\end{proof}

\begin{lemma}\label{main4}
$\sum_{v\in V(G)}(d(v)-4)\ge0.$
\end{lemma}

\begin{proof} We use the discharging method to show that $\sum_{v\in V(G)}(d(v)-4)\ge0$ for any $G\in\mathcal{G}$. Let $v\in V(G)$ have an initial charge of $\mu(v)=d(v)-4$.
We design some discharging rules and redistribute weights accordingly.
Once the discharging is finished, a new weight function $\mu^*(v)$ is
produced, so that the sum of all weights is kept fixed when the
discharging is in process. It is enough to prove that $\mu^*(v)\ge 0$ for all $v\in V(G)$.

The discharging rules are designed as follows:

\begin{enumerate}
    \item[(R1)] Each $4$-vertex gives $\frac{1}{2}$ to each $3$-neighbor.
    \item[(R2)] Each $5$-vertex gives $\frac{1}{3}$ to each $3$-neighbor and $\frac{1}{6}$ to each $4$-neighbor.
    \item[(R3)] Each $d$-vertex with $d\ge 6$ gives $\frac{d-4}{d}$ to each of its neighbors.
\end{enumerate}

Let $v$ be a vertex in $G$. By Lemma~\ref{lem}(i), $d(v)\ge3$. If $d(v)\ge6$, then by (R3) $\mu^*(v)\ge (d(v)-4)-\frac{d(v)-4}{d(v)}\cdot d(v)=0$. If $d(v)=5$, then by (R2) and Lemma~\ref{lem}(iii) $v$ gives $\frac{1}{3}$ to at most one $3$-neighbor and $\frac{1}{6}$ to each $4$-neighbor. So  $\mu^*(v)\ge5-4-\frac{1}{3}-\frac{1}{6}\cdot 4=0$. If $d(v)=4$, then  $v$ does not give out charge and its final charge is $4-4=0$ when $v$ has no $3$-neighbor. So we may assume that $v$ has a $3$-neighbor. Then $v$ has at least three $5^+$-neighbors by Lemma~\ref{lem}(iii), so by (R1), it gets at least $\frac{1}{6}\cdot 3=\frac{1}{2}$ and sends out $\frac{1}{2}$, so $\mu^*(v)\ge4-4-\frac{1}{2}+\frac{1}{6}\cdot3=0$. If $d(v)=3$, then by Lemma~\ref{lem}(ii) $v$ has three $4^+$-neighbors , each of which gives $v$ at least $\frac{1}{3}$. So $\mu^*(v)\ge3-4-\frac{1}{3}\cdot3=0$.
\end{proof}

\begin{proof}[\bf Proof of Theorem~\ref{main1}]
Suppose that $G\in \mathcal{G}$ is a counterexample to Theorem~\ref{main1} such that $|E(G)|$ is as small as possible. Then $G$ has no two edge-disjoint spanning trees but any graph $G'\in \mathcal{G}$ with $|E(G')|<|E(G)|$ has two edge-disjoint spanning trees.

By Theorem~\ref{thm:nash}, if $|E(G)|\ge 2(|V(G)|-1)$, then $G$ has a nontrivial subgraph $H$ with two edge-disjoint spanning trees. Since $G/H$ is still in $\mathcal{G}$, by the minimality of $G$, $G/H$ has two edge-disjoint spanning trees.
Both $H$ and $G/H$ have two edge-disjoint spanning trees, and this implies that $G$ also has two edge-disjoint spanning trees, a contradiction.
Thus we have $|E(G)|< 2(|V(G)|-1)$, and so $|E(G)|\le 2|V(G)|-3$. This implies that $\sum_{v\in V(G)}(d(v)-4)\le -6$, contrary to Lemma~\ref{main4}. 
\end{proof}

To prove Corollary~\ref{cor:line}, we need the following core and trail concepts as well as the corresponding results. For a connected graph $G$ whose line graph $L(G)$ is not a complete graph, the {\it core} of this graph $G$, denoted by $G_0$, is obtained by deleting all the vertices of degree $1$ and contracting exactly one edge $xy$ or $yz$ for each path $xyz$ in $G$ with $d_G(y) = 2$. For any $e_1, e_2\in E(G)$, an {\it $(e_1, e_2)$-trail} is a trail with end edges $e_1$ and $e_2$. A trail $T$ of $G$ is {\em dominating} in $G$ if every edge of $G$ is incident with an internal vertex of $T$. A trail $T$ of $G$ is {\em spanning} if $T$ is dominating in $G$ and $V(T)=V(G)$. We call $G$ {\em spanning trailable} if $G$ has a spanning trail.

\begin{lemma}[Shao~\cite{Shao05}]\label{lem:core}
Let $G$ be a connected and essentially $3$-edge-connected graph. Then
\begin{enumerate}[(i)]
\item $G_0$ is uniquely defined, and $\kappa'(G_0)\ge 3$;
\item if $G_0$ is spanning trailable, then the line graph $L(G)$ is Hamilton-connected.
\end{enumerate}
\end{lemma}


\begin{theorem}[Catlin and Lai~\cite{CL91}]\label{thm:trail}
Let $G$ be a graph with two edge-disjoint spanning trees. For any $e_1, e_2\in E(G)$, $G$ has a spanning $(e_1, e_2)$-trail if and only if $\{e_1, e_2$\} is not an essential edge cut of $G$.
\end{theorem}

\begin{proof}[\bf Proof of Corollary~\ref{cor:line}]
Since $L(G)$ is $5$-connected and essentially $8$-connected, $G$ is essentially $5$-edge-connected and $2$-essentially $8$-edge-connected. Let $G_0$ denote the core of $G$. By the definition of core and Lemma~\ref{lem:core}(i), $G_0$ is 3-edge-connected essentially $5$-edge-connected and $2$-essentially $8$-edge-connected.

By Theorem~\ref{main1}, $G_0$ has two edge-disjoint spanning trees. By Theorem~\ref{thm:trail} and Lemma~\ref{lem:core}(ii), $L(G)$ is Hamilton-connected.
\end{proof}


\begin{thebibliography}{99}


\bibitem{CL91}
P.A. Catlin and H.-J. Lai, Spanning trails joining two given edges, Graph Theory, Combinatorics, and Applications 1 (1991), 207--222.

\bibitem{CS69}
G. Chartrand and M.J. Stewart, The connectivity of line-graphs, Math. Ann. 182 (1969), 170--174.


\bibitem{KV12}
T. Kaiser and P. Vr\'ana, Hamilton cycles in 5-connected line graphs, European J. Combin. 33 (2012), 924--947.

\bibitem{KV21}
T. Kaiser and P. Vr\'ana, The Hamiltonicity of essentially 9-connected line graphs, J. Graph Theory 97 (2021), 241--259.

\bibitem{LL19}
H.-J. Lai and J. Li, Packing spanning trees in highly essentially connected graphs, Discrete Math. 342 (2019), 1--9.

\bibitem{LSWZ06}
H.-J. Lai, Y. Shao, H. Wu and J. Zhou, Every 3-connected, essentially 11-connected line graph is Hamiltonian, J. Combin. Theory Ser. B 96 (2006), 571--576.


\bibitem{LY12}
H. Li and W. Yang, Every 3-connected essentially 10-connected line graph is Hamilton-connected, Discrete Math. 312 (2012), 3670--3674.

\bibitem{Nash61}
C. St. J. A. Nash-Williams, Edge-disjoint spanning trees of finite graphs, J. London Math. Soc. 36 (1961), 445--450.

\bibitem{Nash64}
C. St. J. A. Nash-Williams, Decomposition of finite graphs into forests, J. London Math. Soc. 9 (1964), 12.


\bibitem{RV11}
Z. Ryj\'a\v cek and P. Vr\'ana, Line graphs of multigraphs and Hamilton-connectedness of claw-free graphs, J. Graph Theory 66 (2011), 152--173.

\bibitem{Shao05}
Y. Shao, Claw-free graphs and line graphs, Ph.D. Dissertation, West Virginia University, 2005.

\bibitem{Thom86}
C. Thomassen, Reflections on graph theory, J. Graph Theory 10 (1986), 309--324.

\bibitem{Tutte61}
W. T. Tutte, On the problem of decomposing a graph into $n$ factors, J. London Math. Soc. 36 (1961), 221--230.

\bibitem{YLLG12}
W. Yang, H.-J. Lai, H. Li and X. Guo, Collapsible graphs and Hamiltonian connectedness of line graphs, Discrete Appl. Math. 160 (2012), 1837--1844.

\bibitem{Zhan91}
S. Zhan, On Hamiltonian line graphs and connectivity, Discrete Math. 89 (1991), 89--95.



\end{thebibliography}
\end{document}